\documentclass[11pt, amsfonts]{amsart}

\usepackage{amsmath,amssymb,amsthm}
\usepackage[all]{xy}
\usepackage[dvipdfm,colorlinks=true]{hyperref}

\numberwithin{equation}{section}

\SelectTips{eu}{12}

\newtheorem{theorem}{Theorem}[section]
\newtheorem{proposition}[theorem]{Proposition}
\newtheorem{lemma}[theorem]{Lemma}
\newtheorem{corollary}[theorem]{Corollary}

\theoremstyle{definition}

\newtheorem{example}[theorem]{Example}

\theoremstyle{remark}

\newcommand{\Z}{\mathbb{Z}}
\newcommand{\R}{\mathbb{R}}
\newcommand{\C}{\mathbb{C}}
\newcommand{\F}{\mathbb{F}}
\newcommand{\Ad}{\mathrm{Ad}}

\title[Right-angled Coxeter quandles and polyhedral products]{Right-angled Coxeter quandles\\and polyhedral products}

\author{Daisuke Kishimoto}
\address{Department of Mathematics, Kyoto University, Kyoto, 606-8502, Japan}
\email{kishi@math.kyoto-u.ac.jp}
\begin{document}

\baselineskip.525cm

\maketitle

\begin{abstract}
To a Coxeter group $W$ one associates a quandle $X_W$ from which one constructs a group $\Ad(X_W)$. This group turns out to be an intermediate object between $W$ and the associated Artin group. By using a result of Akita, we prove that $\Ad(X_W)$ is given by a pullback involving $W$, and by using this pullback, we show that the classifying space of $\Ad(X_W)$ is given by a space called a polyhedral product whenever $W$ is right-angled. Two applications of this description of the classifying space are given.
\end{abstract}

%%%%% Section 1 %%%%%

\section{Introduction}

A quandle is a set with a binary operation satisfying three conditions. Since these three conditions are thought of as algebraic abstraction of the Reidemeister moves of knots, a quandle has been intensively studied in low dimensional topology. The three conditions are also thought of as axiomatization of conjugation in a group, so it has been studied in representation theory as well. In this paper, we comprehend a quandle in the latter sense.

We look at the following connection between groups and quandles. Any conjugation closed subset $X$ of a group $G$ can be regarded as a quandle by a binary operation given by conjugation, and to any quandle $Y$ one associates a group $\Ad(Y)$ which is called the adjoint group. Thus any conjugation closed subset $X$ of a group yields a new group $\Ad(X)$.

We now consider a Coxeter group $W$. The set of reflections $X_W$ of $W$ is closed under conjugation, so $X_W$ is a quandle which we call the Coxeter quandle associated with $W$. Then we get a group $\Ad(X_W)$ which is the object to study in this paper. The adjoint group $\Ad(X_W)$ has been studied mainly in connection with representation theory \cite{AFGV,E}, and there are few results on its topology \cite{N}. This paper studies the topology of $\Ad(X_W)$ by applying the recent result of Akita \cite{A}, and in particular, we show that if $W$ is right-angled, then the classifying space of $\Ad(X_W)$ is given by a space which is called a polyhedral product.

The following fundamental property of $\Ad(X_W)$ suggests that $\Ad(X_W)$ possibly gives a new direction for the study of Coxeter groups and Artin groups. The symmetric group $\Sigma_n$ of $n$ letters is a Coxeter group and its associated Artin group is the braid group $B_n$ of $n$ strands. Then there is a natural epimorphism $B_n\to\Sigma_n$. In \cite{AFGV}, it is shown that $\Ad(X_{\Sigma_n})$ is an intermediate object between $\Sigma_n$ and $B_n$ in the sense that the epimorphism $B_n\to\Sigma_n$ factors as the composite of epimorphisms $B_n\to\Ad(X_{\Sigma_n})\to\Sigma_n$. Akita generalized this result to an arbitrary Coxeter group. Let $A_W$ be the Artin group associated with a Coxeter group $W$.

\begin{theorem}
\label{pi}
For an arbitrary Coxeter group $W$, the epimorphism $A_W\to W$ factors as the composite of epimorphisms
$$A_W\to\Ad(X_W)\to W.$$
\end{theorem}

Akita \cite{A} studied further structures of $\Ad(X_W)$ and generalized Eisermann's result \cite{E} on $\Ad(X_{\Sigma_n})$ concerning abelianization. From this, we will deduce that $\Ad(X_W)$ is given by a certain pullback involving $W$ and its abelianization. 

When $W$ is right-angled, it is known that the classifying spaces of $W$ and $A_W$ are given by polyhedral products. Then it is natural to ask whether the classifying space of $\Ad(X_W)$ is a polyhedral product or not, if $W$ is right-angled. We will give an affirmative answer to this question by using the above pullback description of $\Ad(X_W)$. Then we will showe two application of this description of the classifying space: a stable splitting of the classifying space of $\Ad(X_W)$ and calculation of the mod 2 cohomology of $\Ad(X_W)$.

\emph{Acknowledgement:} The author is grateful to Toshiyuki Akito for explaining his work \cite{A} and giving useful comments. Thanks also go to Takefumi Nosaka for comments and to Ye Liu and Mentor Stafa for careful reading of the first draft. The author was partly supported by JSPS KAKENHI (No.\,17K05248).

%%%%% Section 2 %%%%%

\section{Recollection on Coxeter groups}

Recall that a pair $(W,S)$ of a group $W$ and a set $S$ is called a Coxeter system if there is given a map $m\colon S\times S\to\mathbb{N}\cup\{\infty\}$, called the Coxeter matrix, satisfying the following conditions:

\begin{enumerate}
\item $m(s,t)=m(t,s)$ for any $s,t$;
\item $m(s,t)=1$ if and only if $s=t$;
\item $W$ is defined by the presentation
$$W=\langle s\in S\,\vert\,(st)^{m(s,t)}=1\text{ for }m(s,t)<\infty\rangle.$$
\end{enumerate}

We call the group $W$ a Coxeter group, by which we often mean the Coxeter system $(W,S)$ too.

The Artin group associated with a Coxeter system $(W,S)$ is defined by
$$A_W=\langle a_s\;(s\in S)\,\vert\,\underbrace{a_sa_ta_s\cdots}_{m(s,t)}=\underbrace{a_ta_sa_t\cdots}_{m(t.s)}\text{ for }2\le m(s,t)<\infty\rangle.$$
Since the Coxeter group is alnternatively presented as
$$W=\langle s\in S\,\vert\,\underbrace{sts\cdots}_{m(s,t)}=\underbrace{tst\cdots}_{m(t.s)}\text{ for }m(s,t)<\infty\rangle,$$
one gets:

\begin{proposition}
For a Coxeter system $(W,S)$, the assignment 
$$\pi\colon A_W\to W,\quad a_s\mapsto s\quad(s\in S)$$
 is a well-defined epimorphism.
\end{proposition}

\begin{example}
The symmetric group of $n$ letters $\Sigma_n$ is a Coxeter group with the generating set $\{\sigma_1,\ldots,\sigma_{n-1}\}$, where $\sigma_i$ is the transposition $(i\;i+1)$. Indeed, $\Sigma_n$ is generated by $\sigma_1,\ldots,\sigma_{n-1}$ subject to the relations $\sigma_i^2=1$, $\sigma_i\sigma_{i+1}\sigma_i=\sigma_{i+1}\sigma_i\sigma_{i+1}$ for $i=1,\ldots,n-1$ and $\sigma_i\sigma_j=\sigma_j\sigma_i$ for $|i-j|\ge 2$. Then the associated Artin group is the braid group with $n$ strands $B_n$ and the map $\pi\colon B_n\to\Sigma_n$ maps each braid to a permutation given by its ends.
\end{example}

Let $\mathcal{R}_W$ be the complete set of representatives of  $S/\sim$, where $\sim$ is given by conjugation by elements of $W$, and let $c(W)$ be the cardinality of $\mathcal{R}_W$. For instance, we have $c(\Sigma_n)=1$. Let $G_\mathrm{ab}$ denote the abelianization of a group $G$. As in \cite{BMMN}, one has:
 
\begin{proposition}
\label{abelianization}
There are isomorphisms $W_\mathrm{ab}\cong(\Z/2)^{c(W)}$ and $(A_W)_\mathrm{ab}\cong\Z^{c(W)}$.
\end{proposition}

A Coxeter system $(W,S)$ is called right-angled if the Coxeter matrix $m$ satisfies $m(s,t)=1,2,\infty$ for any $s,t\in S$. By the definition of $c(W)$, one gets:

\begin{proposition}
If a Coxeter system $(W,S)$ is right-angled, then $c(W)=|S|$.
\end{proposition}

For a Coxeter system $(W,S)$, we define a graph $\Gamma_W$ such that the vertex set is $S$ and vertices $s,t\in S$ are joined by an edge whenever $2\le m(s,t)<\infty$. Then if a Coxeter group $W$ is right-angled, all the information of $W$ is included in the graph $\Gamma_W$. We will see this more precisely below in terms of a graph product of groups.

%%%%% Section 3 %%%%%

\section{The adjoint group of a Coxeter quandle}

A quandle is a set $X$ with a binary operation $*\colon X\times X\to X$ satisfying the three conditions:
\begin{enumerate}
\item $x*x=x$;
\item $(x*y)*z=(x*z)*(y*z)$;
\item the map $X\to X$, $x\mapsto x*y$ is bijective for any $y\in X$.
\end{enumerate}
A quandle is related with group theory (and representation theory) as follows. For a group $G$, we put $x*y=y^{-1}xy$ for $x,y\in G$. Then one can easily check the three conditions of a quandle so that $G$ is a quandle with this binary operation. More generally, any conjugation closed subset of a group can be regarded as a quandle in the same way.

Motivated by the above construction of a quandle from a group, for a quandle $X$, we define a group
$$\Ad(X)=\langle e_x\;(x\in X)\,\vert\,e_{x*y}=e_y^{-1}e_xe_y\rangle$$
which is called the adjoint group of $X$. When $X$ is a conjugation closed subset of a group $G$ regarded as a quandle, $\Ad(X)$ is directly related with $G$.

\begin{proposition}
\label{phi}
For a conjugation closed subset $X$ of a group $G$, the assignment 
$$\phi\colon\Ad(X)\to G,\quad e_x\mapsto x\quad(x\in X)$$
is a well-defined homomorphism. Moreover, if $X$ generates $G$, then $\phi$ is surjective.
\end{proposition}

\begin{proof}
For $x,y\in X$, we have $\phi(e_{x*y})=x*y=y^{-1}xy=\phi(e_y^{-1}e_xe_y)$, implying that the map $\phi$ is a well-defined homomorphism. The remaining statement is obvious.
\end{proof}

Let $(W,S)$ be a Coxeter system. An element of $W$ of the form $w^{-1}sw$ for $w\in W$ and $s\in S$ is called a reflection of $W$. Then the set of reflections of $W$, denoted $X_W$, is a conjugation closed subset of $W$, so $X_W$ is a quandle which is called the Coxeter quandle associated with a Coxeter system $(W,S)$. Since $W$ is generated by $X_W$, we have the following by Proposition \ref{phi}.

\begin{corollary}
\label{phi_Coxeter}
For a Coxeter group $W$, the map $\phi\colon\Ad(X_W)\to W$ is an epimorphism.
\end{corollary}

Akita \cite{A} showed that $\Ad(X_W)$ is related also with the Artin group $A_W$, where we reproduce its proof.

\begin{proposition}
\label{Phi}
For an arbitrary Coxeter system $(W,S)$, the assignment 
$$\Phi\colon A_W\to\Ad(X_W),\quad a_s\mapsto e_s\quad(s\in S)$$
is a well-defined epimorphism.
\end{proposition}

\begin{proof}
Suppose that $m(s,t)=k$ with $2\le k<\infty$ for $s,t\in S$. Then we have
$$\underbrace{(\cdots(s*t)*s\cdots)}_k=\underbrace{\cdots s^{-1}t^{-1}}_{k-1}\underbrace{sts\cdots}_k=\underbrace{\cdots s^{-1}t^{-1}}_{k-1}\underbrace{tst\cdots}_k=u$$
where $u=s$ for $n$ even and $u=t$ for $n$ odd, that is, $u$ is the last letter of the word $\underbrace{tst\cdots}_k$. Then it follows that
$$\Phi(\underbrace{a_sa_ta_s\cdots}_k)=\underbrace{e_se_te_s\cdots}_k=e_te_{s*t}\underbrace{e_te_s\cdots}_{k-1}=\underbrace{e_te_se_t\cdots}_{k-1}e_{\scriptsize\underbrace{\cdots((s*t)*s)*\cdots}_k}=\underbrace{e_te_se_t\cdots}_k,$$
implying that $\Phi$ is a well-defined homomorphism. To see that $\Phi$ is surjective, we shall show that $\Ad(X_W)$ is generated by $e_s$ for $s\in S$. For $s_1,\ldots,s_n,s\in S$ and $w=s_1\cdots s_n$, one has $w^{-1}sw=(\cdots(s*s_1)*s_2\cdots)*s_n$, implying that $e_{w^{-1}sw}=e_{\cdots(s*s_1)*s_2\cdots)*s_n}=e_{s_n}^{-1}\cdots e_{s_1}^{-1}e_se_{s_1}\cdots e_{s_n}$. Thus the proof is completed.
\end{proof}

Combining Corollary \ref{phi_Coxeter} and Proposition \ref{Phi}, we obtain Theorem \ref{pi}.

We now recall the structure theorem of $\Ad(X_W)$ due to Akita \cite{A}. Eisermann \cite{E} showed that there is a short exact sequence $1\to A_n\to\Ad(X_{\Sigma_n})\to\Z\to 0$, where $A_n$ is the alternating group of $n$ letters. Note that $A_n\cong[\Sigma_n,\Sigma_n]$ and $c(\Sigma_n)=1$. Akita \cite{A} generalized this result to an arbitrary Coxeter group.

\begin{theorem}
\label{Akita}
For any Coxeter group $W$, the following hold:
\begin{enumerate}
\item there is an isomorphism $\Ad(X_W)_\mathrm{ab}\cong\Z^{c(W)}$;
\item the map $\phi\colon\Ad(X_W)\to W$ induces an isomorphism $[\Ad(X_W),\Ad(X_W)]\cong[W,W]$.
\end{enumerate}
\end{theorem}

\begin{corollary}
For any Coxeter group $W$, there is a short exact sequence
$$1\to[W,W]\to\Ad(X_W)\xrightarrow{\mathrm{ab}}\Z^{c(W)}\to 0.$$
\end{corollary}

By \cite{BMMN}, $W_\mathrm{ab}$ is a $\Z/2$-vector space with a basis $\{[x]\in W_\mathrm{ab}\,\vert\,x\in\mathcal{R}_W\}$, from which the first isomorphism of Proposition \ref{abelianization} follows. On the other hand, Akita \cite{A} showed that $\Ad(X_W)_\mathrm{ab}$ is a free abelian group  with a basis $\{[e_x]\in\Ad(X_W)_\mathrm{ab}\,\vert\,x\in\mathcal{R}_W\}$, which yields Theorem \ref{Akita} (1). Then we identify the abelianization of the map $\phi\colon\Ad(X_W)\to W$ as:

\begin{lemma}
The map $\phi_\mathrm{ab}\colon\Ad(X_W)_\mathrm{ab}\to W_\mathrm{ab}$ is identified with the canonical projection $\Z^{c(W)}\to(\Z/2)^{c(W)}$.
\end{lemma}

Thus we get a commutative diagram
\begin{equation}
\label{pb_Ad}
\xymatrix{\Ad(X_W)\ar[r]^{\mathrm{ab}}\ar[d]^\phi&\Z^{c(W)}\ar[d]^{\mathrm{proj}}\\
W\ar[r]^{\mathrm{ab}}&(\Z/2)^{c(W)}.}
\end{equation}
We show that this diagram is a pullback.

\begin{lemma}
\label{pullback_lem}
Suppose that there is a commutative square of groups
\begin{equation}
\label{pb}
\xymatrix{G_1\ar[r]^{f_1}\ar[d]^g&H_1\ar[d]^h\\
G_2\ar[r]^{f_2}&H_2}
\end{equation}
where $f_1$ are surjective. Then the square is a pullback if and only if the canonical map $\mathrm{Ker}\,f_1\to\mathrm{Ker}\,f_2$ is an isomorphism.
\end{lemma}

\begin{proof}
Let $G=\{(x,y)\in G_2\times H_1\,\vert\,f_2(x)=h(y)\}$ which is the pullback of the triad $G_2\xrightarrow{f_2}H_2\xleftarrow{h}H_1$. Let $p_1\colon G\to G_2$ and $p_2\colon G\to H_1$ be the projections, and define a map $e\colon G_1\to G$ by $e(x)=(g(x),f_1(x))$ for $x\in G_1$. Then we have $p_1\circ e=g$ and since $f_1$ is surjective, the projection $p_2$ is surjective too. Moreover, the commutative square \eqref{pb} extends to the following commutative diagram with exact columns and rows
$$\xymatrix{1\ar[r]&\mathrm{Ker}\,f_1\ar[r]\ar[d]^{\bar{e}}&G_1\ar[r]^{f_1}\ar[d]^e&H_1\ar[r]\ar@{=}[d]&1\\
1\ar[r]&\mathrm{Ker}\,p_2\ar[r]\ar[d]^{\bar{p}_1}&G\ar[r]^{p_2}\ar[d]^{p_1}&H_1\ar[r]\ar[d]^h&1\\
1\ar[r]&\mathrm{Ker}\,f_2\ar[r]&G_2\ar[r]^{f_2}&H_2}$$
The square diagram \eqref{pb} is a pullback if and only if the map $e\colon G_1\to G$ is an isomorphism. By the above diagram, the latter is equivalent to that the map $\bar{e}$ is an isomorphism. One sees that $\mathrm{Ker}\,p_2\cong\mathrm{Ker}\,f_2$ and the map $\bar{p}_1$ is identified with the identity map. Thus the map $\bar{e}$ is identified with the canonical map $\mathrm{Ker}\,f_1\to\mathrm{Ker}\,f_2$, completing the proof.
\end{proof}

\begin{theorem}
\label{pullback_thm}
The commutative square \eqref{pb_Ad} is a pullback.
\end{theorem}

\begin{proof}
Combine Theorem \ref{Akita} (2) and Lemma \ref{pullback_lem}.
\end{proof}

In \cite{AFGV} it is proved that there is a short exact sequence $0\to\Z\to\Ad(X_{\Sigma_n})\to\Sigma_n\to 1$, and this was also generalized by Akita \cite{A} to an arbitrary Coxeter group. We can reprove this by applying Theorem \ref{pullback_thm}

\begin{corollary}
\label{exact}
For any Coxeter group $W$, there is a short exact sequence
$$0\to\Z^{c(W)}\to\Ad(X_W)\xrightarrow{\phi}W\to 1.$$
\end{corollary}

\begin{proof}
By Theorem \ref{pullback_thm}, the kernel of $\phi$ is isomorphic to the kernel of the projection $\Z^{c(W)}\to(\Z/2)^{c(W)}$, which is isomorphic to $\Z^{c(W)}$. Thus the proof is done.
\end{proof}

%%%%% Section 4 %%%%%

\section{Classifying spaces and polyhedral products}
 
 Let $K$ be a simplicial complex on the vertex set $[m]=\{1,\ldots,m\}$ and $(X,A)$ be topological pair. The polyhedral product of $(X,A)$ with respect to $K$ is defined by
$$Z(K;(X,A))=\bigcup_{\sigma\in K}(X,A)^\sigma$$
where $(X,A)^\sigma=Y_1\times\cdots\times Y_m$ such that $Y_i=X,A$ according to $i\in\sigma$ and $i\not\in\sigma$. Note that $Z(K;(X,A))$ is natural with respect to $(X,A)$ and inclusions of simplicial complexes. Polyhedral products were introduced as a generalisation of the moment-angle complex and the Davis-Januszkiewicz space which are fundamental in toric topology \cite{BBCG,DJ}, and connect algebraic geometry, combinatorics, commutative algebra, geometry, group theory, and topology. Among others, their homotopy theory is rapidly developing (cf. \cite{IK2}). 

We will use the following property of polyhedral products, which is immediately deduced from the definition. For $\emptyset\ne I\subset[m]$, the full subcomplex of $K$ on $I$ is defined by $K_I=\{\sigma\in K\,\vert\,\sigma\subset I\}$. 

\begin{proposition}
\label{retract}
For $\emptyset\ne I\subset[m]$, $Z(K_I;(X,A))$ is a retract of $Z(K;(X,A))$.
\end{proposition}

For an acyclic space $X$, the acyclicity of $Z(K;(X,*))$ is completely characterized in terms of $K$. The necessity of the acyclicity of $Z(K;(X,*))$ has been often referred to \cite{DO} although it is an easy consequence of the old result of Whitehead \cite{W}. So we here give a simple proof using the result of Whitehead. A simplicial complex $K$ is called flag if $\sigma\subset[m]$ is a simplex of $K$ whenever any two elements of $\sigma$ are joined by an edge of $K$. 

\begin{proposition}
\label{acyclic}
For an acyclic space $X$, $Z(K;(X,*))$ is acyclic if and only if $K$ is flag.
\end{proposition}

\begin{proof}
Whitehead \cite{W} proved the following. Suppose that there is a homotopy pushout of path-connected spaces
$$\xymatrix{X_{12}\ar[r]^{\alpha_1}\ar[d]^{\alpha_2}&X_1\ar[d]\\
X_2\ar[r]&X}$$
in which $X_1,X_2,X_{12}$ are acyclic and $\alpha_1,\alpha_2$ are injective in $\pi_1$. Then $X$ is acyclic. We apply this to prove the if part of the proposition.

For a vertex $v$ of $K$, let $\mathrm{lk}(v)$ and $\mathrm{dl}(v)$ denote the link and the deletion of a vertex $v$ in $K$. Then there is a pushout of simplicial complexes
$$\xymatrix{\mathrm{lk}(v)\ar[r]\ar[d]&\mathrm{lk}(v)*\{v\}\ar[d]\\
\mathrm{dl}(v)\ar[r]&K}$$
which induces a homotopy pushout
\begin{equation}
\label{po_Z}
\xymatrix{Z(\mathrm{lk}(v);(X,*))\ar[r]\ar[d]&Z(\mathrm{lk}(v);(X,*))\times X\ar[d]\\
Z(\mathrm{dl}(v);(X,*))\ar[r]&Z(K;(X,*)).}
\end{equation}
The upper horizontal arrow is obviously injective in $\pi_1$. $K$ is flag if and only if $K_V=\mathrm{lk}(v)$ for any vertex $v$, where $V$ is the vertex set of $\mathrm{lk}(v)$. Then it follows from Proposition \ref{retract} that if $K$ is flag, the left vertical arrow of \eqref{po_Z} is injective in $\pi_1$. Moreover, if $K$ is flag, then $K_I$ is flag too for any $\emptyset\ne I\subset[m]$. Then we can apply the above result of Whitehead to \eqref{po_Z} inductively on the number of vertices. Thus we obtain that if $K$ is flag, then $Z(K;(X,*))$ is acyclic for an acyclic space $X$.

Next we conversely suppose that $Z(K;(X,*))$ is acyclic for an acyclic space $X$. Assume that $K$ is not flag. Then there is $I\subset[m]$ with $|I|\ge 3$ such that $K_I$ is the boundary of the $(|I|-1)$-dimensional simplex. Recall from \cite{IK2} that there are a homotopy fibration $Z(K_I;(C\Omega X,\Omega X))\to Z(K_I;(X,*))\to X^{|I|}$ and a homotopy equivalence $Z(K_I;(C\Omega X,\Omega X))\simeq\Sigma^{|I|-1}\Omega X\wedge\cdots\wedge\Omega X$, where the number of $\Omega X$ in the wedge is $|I|$. Since $\Omega X$ is discrete, $Z(K_I;(C\Omega X,\Omega X))$ is a wedge of spheres of dimension $|I|-1\ge 2$. Then $Z(K_I;(X,*))$ is not acyclic. On the other hand, by Proposition \ref{retract}, $Z(K_I;(X,*))$ is a retract of $Z(K;(X,*))$, implying that $Z(K_I;(X,*))$ is acyclic. This is a contradiction, so $K$ is flag. Therefore the proof is completed.
\end{proof}

Let $G$ be a group and $\{G_s\}_{s\in S}$ be a family of groups such that $G_s=G$ for all $s\in S$. We denote the free product of $G_s$ for $s\in S$ by $F_S(G)$. Let $\Gamma$ be a simple graph (a graph without loops and multiple edges) with the vertex set $S$. The graph product of $G$ with respect to $\Gamma$, denoted $G^\Gamma$, is defined by dividing out $F_S(G)$ by the commuting relations $[G_s,G_t]=1$ for edges $\{s,t\}$ of $\Gamma$. Note that $G^\Gamma$ is natural with respect to homomorphisms of groups and inclusions of graphs.

\begin{lemma}
\label{pi_1}
For a path-connected space $X$, there is an isomorphism
$$\pi_1(Z(K;(X,*)))\cong\pi_1(X)^{K^{(1)}}$$
where $K^{(n)}$ denotes the $n$-skeleton of $K$.
\end{lemma}

\begin{proof}
By the cellular approximation theorem, the inclusion $Z(K^{(1)};(X,*))\to Z(K;(X,*))$ is an isomorphism in $\pi_1$. Since $Z(K^{(0)};(X,*))$ is a wedge of $m$ copies of $X$, we have $\pi_1(Z(K^{(0)};(X,*)))\cong F_{[m]}(\pi_1(X))$. By the van Kampen theorem, attaching the edge $\{i,j\}$ adds the commutator relation of $i$-th and $j$-th $\pi_1(X)$ in $F_{[m]}(\pi_1(X))$. Thus we have proved the lemma.
\end{proof}

For a simple graph $\Gamma$, let $C(\Gamma)$ be the flag complex whose 1-skeleton is $\Gamma$.

\begin{proposition}
\label{BG_Z}
For a group $G$ and a finite simple graph $\Gamma$, there is a homotopy equivalence
$$B(G^\Gamma)\simeq Z(C(\Gamma);(BG,*))$$
which is natural with respect to $G$ and $\Gamma$.
\end{proposition}

\begin{proof}
The homotopy equivalence is obtained by Proposition \ref{acyclic} and Lemma \ref{pi_1}, and the naturality is obvious by the construction.
\end{proof}

We now consider Coxeter groups. If a Coxeter group $W$ is right-angled, there are isomorphisms $W\cong(\Z/2)^{\Gamma_W}$ and $A_W\cong\Z^{\Gamma_W}$. Then by Proposition \ref{BG_Z}, one gets:

\begin{corollary}
\label{W_Z}
If a Coxeter system $(W,S)$ is finitely generated and right-angled, then there are homotopy equivalences
$$BW\simeq Z(C(\Gamma_W);(\R P^\infty,*))\quad\text{and}\quad BA_W\simeq Z(C(\Gamma_W);(S^1,*)).$$
\end{corollary}

We are going to show that $\Ad(X_W)$ is also given by a certain polyhedral product when $W$ is right-angled. To this end, we need several lemmas. The following lemma is well known and useful, which is proved, for example, in \cite[Proposition, pp.180]{F}.

\begin{lemma}
\label{hocolim}
Let $\{F_i\to E_i\to B\}_{i\in I}$ be an $I$-diagram of homotopy fibration with a fixed base $B$. Then the sequence
$$\underset{I}{\mathrm{hocolim}}\,F_i\to\underset{I}{\mathrm{hocolim}}\,E_i\to B$$
is a homotopy fibration.
\end{lemma}

The following is an up-to-homotopy version of \cite[Lemma 2.3.1]{DS}.

\begin{lemma}
\label{homotopy_fib}
Let $(F,F')\to(E,E')\to(B,B)$ be a pair of homotopy fibrations such that $(F,F'),(E,E')$ are NDR pairs. Then
$$Z(K;(F,F'))\to Z(K;(E,E'))\to B^m$$
is a homotopy fibration.
\end{lemma}

\begin{proof}
For any $\sigma\subset[m]$, the sequence $(F,F')^\sigma\to(E,E')^\sigma\to B^m$ is a homotopy fibration, where $(X,A)^\sigma$ is as in the definition of a polyhedral product in the previous section. Since this homotopy fibration is natural with respect to $\sigma$, we get a $F(K)$-diagram of homotopy fibrations $\{(F,F')^\sigma\to(E,E')^\sigma\to B^m\}_{\sigma\in F(K)}$, where $F(K)$ is the face poset of $K$. Then it follows from Lemma \ref{hocolim} that the sequence
$$\underset{F(K)}{\mathrm{hocolim}}\,(F,F')^\sigma\to \underset{F(K)}{\mathrm{hocolim}}\,(E,E')^\sigma\to B^m$$
is a homotopy fibration. Since $(F,F')$ is an NDR pair, the projection $\mathrm{hocolim}\,(F,F')^\sigma\to\mathrm{colim}\,(F,F')^\sigma=Z(K;(F,F'))$ is a homotopy equivalence. Similarly we get a natural homotopy equivalence $\mathrm{hocolim}\,(E,E')^\sigma\simeq Z(K;(E,E'))$. Thus the proof is completed.
\end{proof}

\begin{lemma}
\label{pb_lem}
Let $F\to E\to B$ be a homotopy fibration such that $F\to E$ is a cofibration. Then the following commutative square is a homotopy pullback.
\begin{equation}
\label{pb_fib}
\xymatrix{Z(K;(E,F))\ar[r]\ar[d]&E^m\ar[d]\\
Z(K;(B,*))\ar[r]&B^m}
\end{equation}
\end{lemma}

\begin{proof}
To see that \eqref{pb_fib} is a homotopy pullback, it is sufficient to show that the natural map between the homotopy fibers of the horizontal arrows is a homotopy equivalence. By Lemma \ref{homotopy_fib}, the homotopy fibers of the both horizontal arrows are homotopy equivalent to $Z(K;(C\Omega B,\Omega B))$ and the natural map between them is identified with the identity map. Thus \eqref{pb_fib} is a homotopy pullback.
\end{proof}

\begin{lemma}
\label{pb_group}
Suppose that there is a pullback of groups
$$\xymatrix{G_1\ar[r]^g\ar[d]&G_2\ar[d]\\
H_1\ar[r]^h&H_2}$$
where $g,h$ are surjective. Then the induced square
$$\xymatrix{BG_1\ar[r]\ar[d]&BG_2\ar[d]\\
BH_1\ar[r]&BH_2}$$
is a homotopy pullback.
\end{lemma}

\begin{proof}
It is sufficient to show that the natural map between the homotopy fibers of the horizontal arrows in the second square is a homotopy equivalence. Since $g,h$ are surjective, the homotopy fibers are $B\mathrm{Ker}\,g$ and $B\mathrm{Ker}\,h$, respectively, and the natural map between them is induced from the canonical map $\mathrm{Ker}\,g\to\mathrm{Ker}\,h$ which is an isomorphism by Lemma \ref{pullback_lem}. Thus the proof is completed.
\end{proof}

Let $M$ be the closed M\"obius band and $(M,S^1)$ be the pair of $M$ and its boundary circle.

\begin{theorem}
\label{main}
For a finitely generated right-angled Coxeter group $W$, there is a homotopy equivalence
$$B\Ad(X_W)\simeq Z(C(\Gamma_W);(M,S^1)).$$
\end{theorem}

\begin{proof}
By Theorem \ref{pullback_thm} and Lemma \ref{pb_group}, for any Coxeter group $W$, we have a homotopy pullback
$$\xymatrix{B\Ad(X_W)\ar[r]\ar[d]&(S^1)^{c(W)}\ar[d]\\
BW\ar[r]&(\R P^\infty)^{c(W)}.}$$
Consider a homotopy fibration $S^1\to M\to\R P^\infty$ where the first map is the boundary inclusion and the second map is equivalent to the bottom cell inclusion. Then by Lemma \ref{pb_lem}, there is a homotopy pullback
$$\xymatrix{Z(K;(M,S^1))\ar[r]\ar[d]&(S^1)^m\ar[d]\\
Z(K;(\R P^\infty,*))\ar[r]&(\R P^\infty)^m.}$$
Thus the proof is done by Corollary \ref{W_Z}.
\end{proof}

By Corollary \ref{W_Z} and Theorem \ref{main} together with the naturality of Proposition \ref{BG_Z}, one gets:

\begin{corollary}
\label{phi_Phi}
Let $W$ be a finitely generated right-angled Coxeter group.
\begin{enumerate}
\item The map $\Phi\colon BA_W\to B\Ad(X_W)$ is identified with
$$Z(C(\Gamma_W);(S^1,*))\to Z(C(\Gamma_W);(M,S^1))$$
which is induced from the composite $(S^1,*)\simeq(M,*)\to(M,S^1)$.
\item The map $\phi\colon B\Ad(X_W)\to BW$ is identified with
$$Z(C(\Gamma_W);(M,S^1))\to Z(C(\Gamma_W);(\R P^\infty,*))$$
which is induced from the composite $(M,S^1)\to(M/S^1,*)=(\R P^2,*)\to(\R P^\infty,*)$.
\end{enumerate}
\end{corollary}

%%%%% Section 5 %%%%%

\section{Applications}

We give two applications of Theorem \ref{main} and first show a stable splitting of $B\Ad(X_W)$. We will use the following stable splitting of a polyhedral product, which was proved in \cite{BBCG} (See \cite{IK1} for a more precise proof of the naturality). Define the polyhedral smash product $\widehat{Z}(K;(X,A))$ in the same way as $Z(K;(X,A))$ by replacing the direct product with the smash product.

\begin{theorem}
\label{BBCG}
There is a homotopy equivalence
$$\Sigma Z(K;(X,A))\simeq\Sigma\bigvee_{\emptyset\ne I\subset[m]}\widehat{Z}(K_I;(X,A))$$
which is natural with respect to $(X,A)$.
\end{theorem}

\begin{lemma}
\label{split}
The inclusion $Z(K;(X,*))\to Z(K;(X,A))$ has a left homotopy inverse after a suspension.
\end{lemma}

\begin{proof}
By definition, if $L$ is a full simplex with $n$ vertices, then $\widehat{Z}(L;(Y,B))$ is the smash product of $n$ copies of $Y$. Then $\widehat{Z}(K_I;(X,*))=\widehat{Z}(K_I;(X,A))$ if $I\subset[m]$ is a simplex of $K$. On the other hand, if $I$ is not a simplex of $K$, we have $\widehat{Z}(K_I;(X,*))=*$. Thus by Theorem \ref{BBCG}, $\Sigma Z(K;(X,*))$ is a wedge summand of $\Sigma Z(K;(X,A))$, up to homotopy, completing the proof.
\end{proof}

\begin{theorem}
For a finitely generated right-angle Coxeter group W,  the map $\Phi\colon BA_W\to B\Ad(X_W)$ has a left homotopy inverse after a suspension. In particular, there is a simply connected space $X$ such that
$$\Sigma B\Ad(X_W)\simeq\Sigma BA_W\vee\Sigma X.$$
\end{theorem}

\begin{proof}
By Corollary \ref{phi_Phi} and Lemma \ref{split}, the first assertion follows. If we put $X$ to be the cofiber of $\Phi\colon BA_W\to B\Ad(X_W)$, then the second assertion follows from the first one, where $X$ is simply connected by the van Kampen theorem. 
\end{proof}

Let $W$ be a right-angled Coxeter group with the generating set $[m]$. We next calculate the mod 2 cohomology of $\Ad(X_W)$. By Corollary \ref{exact}, there is a homotopy fibration
\begin{equation}
\label{fibration_W}
(S^1)^{m}\to B\Ad(X_W)\to BW.
\end{equation}
We calculate the Serre spectral sequence of \eqref{fibration_W}, where we denote its $E_r$ by $E_r(W)$. For a subset $T\subset[m]$, let $W_T$ be the subgroup of $W$ generated by $T$. Then $W_T$ is a right-angled Coxeter group with the generating set $T$ such that $\Gamma_{W_T}=(\Gamma_W)_T$. By the naturality of Theorem \ref{pullback_thm}, one sees that for $V\subset U\subset[m]$, \eqref{fibration_W} for $W_V$ is a retract of \eqref{fibration_W} for $W_U$, implying the following.

\begin{lemma}
\label{retract}
For $V\subset U\subset[m]$, $E_r(W_V)$ is a retract of $E_r(W_U)$.
\end{lemma}

By Corollary \ref{W_Z}, one has $BW\simeq Z(C(\Gamma_W);(\R P^\infty,*))$. Then quite similarly to the calculation of the cohomology of $Z(K;(\C P^\infty,*))$ in \cite{DJ}, we see that
$$H^*(BW;\F_2)=\F_2[x_1,\ldots,x_m]/(x_ix_j\,\vert\,\{i,j\}\not\in E(\Gamma_W)),\quad|x_i|=1$$
where $E(\Theta)$ denotes the edge set of a graph $\Theta$.  By Theorem \ref{pullback_thm}, the homotopy fibration \eqref{fibration_W} is a homotopy pullback of the homotopy fibration $(S^1)^m\xrightarrow{2}(S^1)^m\to(\R P^\infty)^m$, implying that its local system of coefficients is trivial. Then one gets
$$E_2(W)\cong \F_2[x_1,\ldots,x_m]/(x_ix_j\,\vert\,\{i,j\}\not\in E(\Gamma_W))\otimes\Lambda(y_1,\ldots,y_m)$$
such that $d_2x_i=0$ and $d_2y_i=x_i^2$, where $|y_i|=1$. For $I\subset[m]$, put $z_{i,I}=x_iy_I$, where $y_I=\prod_{i\in I}y_i$ for $I\ne\emptyset$ and $y_\emptyset=1$. Then for $I\subset N_i$, we have $d_2z_{i,I}=0$ and $E_3(W)$ is generated by such $z_{i,I}$ as an algebra, where $N_i=\{j\,\vert\,\{i,j\}\not\in E(\Gamma_W)\}$. By definition, we have
\begin{align}
\label{rel1}
z_{i,I}z_{j,J}&=\begin{cases}0&\{i,j\}\not\in E(\Gamma_W)\text{ or }I\cap J\ne\emptyset\\z_{i,I-k}z_{j,J\sqcup\{k\}}&\{i,j\}\in E(\Gamma_W),\;I\cap J=\emptyset,\;k\in N_j\end{cases}
\end{align}
in $E_2(W)$. Moreover, since $d_2y_I=\sum_{i\in I}z_{i,\emptyset}z_{i,I-i}$, we have
\begin{equation}
\label{rel2}
\sum_{i\in I}z_{i,\emptyset}z_{i,I-i}=0\quad\text{if }I-i\subset N_i\text{ for any }i\in I
\end{equation}
in $E_3(W)$. Since these are all relations that $z_{i,I}\in E_3(W)$ satisfy, one gets:

\begin{lemma}
\label{E_3}
$E_3(W)$ is an algebra generated by $z_{i,I}$ for $i=1,\ldots,m$ and $I\subset N_i$ subject to the relations \eqref{rel1} and \eqref{rel2}, where $|z_{i,I}|=i+|I|$.
\end{lemma}

We show that $E_r(W)$ collapses for $r=3$ by considering the special case that $\Gamma_W=K_{m-1}\sqcup\{m\}$, where $K_n$ is the complete graph with $n$ vertices. If $\Gamma_W=K_{m-1}\sqcup\{m\}$, then $C(\Gamma_W)=\Delta^{[m-1]}\sqcup\{m\}$, where $\Delta^S$ is the full simplex on the vertex set $S$.

\begin{lemma}
\label{dim}
For a right-angled Coxeter group $W$ with $\Gamma_W=K_{m-1}\sqcup\{m\}$, the dimension of $H^m(\Ad(X_W);\F_2)$ is $m+1$.
\end{lemma}

\begin{proof}
By Theorems \ref{main} and \ref{BBCG}, one has
$$\Sigma B\Ad(X_W)\simeq\Sigma\bigvee_{\emptyset\ne I\subset[m]}\widehat{Z}(C(\Gamma_W)_I;(M,S^1)).$$
Note that $C(\Gamma_W)_I=\Delta^I$ for $m\not\in I$ and $C(\Gamma_W)_I=\Delta^{I-m}\sqcup\{m\}$ for $m\in I$. By definition, $\widehat{Z}(\Delta^I;(M,S^1))=\bigwedge_{i\in I}M\simeq S^{|I|}$. When $m\in I$, we have $\widehat{Z}(\Delta^{I-m}\sqcup\{m\};(M,S^1))=U\cup V$ for $U=(\bigwedge_{i\in I-m}M)\wedge S^1$ and $V=(\bigwedge_{i\in I-m}S^1)\wedge M$ such that $U\cap V=\bigwedge^{|I|}S^1=S^{|I|}$. Since the inclusions of $U\cap V$ into $U$ and $V$ are trivial in the mod 2 cohomology, by the Mayer-Vietoris sequence of this cover, we get
$$H^*(\widehat{Z}(\Delta^{I-m}\sqcup\{m\};(M,S^1));\F_2)\cong\begin{cases}\F_2&*=0,|I|+1\\\F_2\oplus\F_2&*=|I|\\0&\text{otherwise}.\end{cases}$$
Thus $H^m(\Ad(X_W);\F_2)\cong\bigoplus_{J\subset[m-1],\,|J|\ge m-2}H^m(\widehat{Z}(\Delta^J\sqcup\{m\};(M,S^1));\F_2)\cong\F_2^{m+1}$, completing the proof.
\end{proof}

Suppose that $\Gamma_W=\Delta^{[m-1]}\sqcup\{m\}$. By Lemma \ref{E_3}, $E_3^m(W)$ is spanned by $z_{m,[m-1]}$, $z_{1,\emptyset}\cdots z_{m-2,\emptyset}z_{m-1,\{m\}}$ and $z_{m,\emptyset}z_{m,[m-1]-i}$ for $i=1,\ldots,m-1$. Then it follows from Lemma \ref{dim} that these elements are permanent cycles. Thus one gets:

\begin{lemma}
\label{permanent}
If $\Gamma_W=\Delta^{[m-1]}\sqcup\{m\}$, then $z_{m,[m-1]}\in E_3(W)$ is a permanent cycle.
\end{lemma}

Let $W$ be an arbitrary right-angled Coxeter group on the vertex set $[m]$. Let $I\subset N_i$ and $W'$ be the right-angled Coxeter group with the generating set $I\sqcup\{i\}$ such that $\Gamma_{W'}=K_{|I|}\sqcup\{i\}$. Since $\Gamma_{W_{I\sqcup\{i\}}}\subset\Gamma_{W'}$, by the naturality of graph products of groups, there is a homotopy commutative diagram
$$\xymatrix{(S^1)^{|I|+1}\ar[r]\ar@{=}[d]&B\Ad(X_{W_{I\sqcup\{i\}}})\ar[r]\ar[d]&BW_{I\sqcup\{i\}}\ar[d]\\
(S^1)^{|I|+1}\ar[r]&B\Ad(X_{W'})\ar[r]&BW'.}$$
Then we get a map $E_r(W')\to E_r(W_{I\sqcup\{i\}})$ which maps $z_{i,I}\in E_3(W')$ to $z_{i,I}\in E_3(W_{I\sqcup\{i\}})$, so by Lemma \ref{permanent}, $z_{i,I}\in E_3(W_{I\sqcup\{i\}})$ is a permanent cycle. Thus by Lemma \ref{retract}, $z_{i,I}\in E_3(W)$ is a permanent-cycle too and we obtain:

\begin{lemma}
\label{collapse}
The Serre spectral sequence of \eqref{fibration_W} collapses at $E_3$.
\end{lemma}

For a degree reason, the extension of $E_\infty(W)$ to $H^*(\Ad(X_W);\F_2)$ is trivial. Therefore by Lemmas \ref{E_3} and \ref{collapse}, we finally obtain:

\begin{theorem}
For a right-angled Coxeter group $W$ with the generating set $[m]$, the mod 2 cohomology of $\Ad(X_W)$ is an algebra generated by $z_{i,I}$ for $i=1,\ldots,m$ and $I\subset N_i$ subject to the relations \eqref{rel1} and \eqref{rel2}, where $|z_{i,I}|=i+|I|$.
\end{theorem}


\begin{thebibliography}{W}
\bibitem{A}T. Akita, \emph{The adjoint group of a Coxeter quandle}, \url{arXiv:1702.07104v2}.
\bibitem{AFGV}N. Andruskiewitsch, F. Fantino, G. A. Garc\'{i}a, and L. Vendramin, \emph{On Nichols algebras associated to simple racks, Groups, algebras and applications}, Contemp. Math. \textbf{537}, Amer. Math. Soc., Providence, RI, 2011, pp. 31-56.
\bibitem{BBCG}A. Bahri, M. Bendersky, F.R. Cohen, and S. Gitler, \emph{The polyhedral product functor: a method of decomposition for moment-angle complexes, arrangements and related spaces}, Advances in Math. \textbf{225} (2010), 1634-1668.
\bibitem{BMMN}N. Brady, J.P. McCammond, B. M\"uhlherr, and W.D. Neumann, \emph{Rigidity of Coxeter groups and Artin groups}, Geom. Dedicata \textbf{94} (2002), 91-109.
%\bibitem[BP]{BP}V.M. Buchstaber and T.E. Panov, {\it Torus actions and their applications in topology and combinatorics}, University Lecture Series {\bf 24}, American Mathematical Society, Providence, RI, 2002. 
%\bibitem{D}M.W. Davis, {\it The geometry and topology of Coxeter groups}, London Mathematical Society Monographs Series {\bf 32}, Princeton University Press, Princeton, NJ, 2007.
\bibitem{DJ}M.W. Davis and T. Januszkiewicz, \emph{Convex polytopes, Coxeter orbifolds and torus actions}, Duke Math. J. \textbf{62} (1991), 417-451.
\bibitem{DO}M. Davis and B. Okun, \emph{Cohomology computations for Artin groups, Bestvina-Brady groups,
and graph products}, Groups Geom. Dyn. \textbf{6} (3) (2012), 485-531.
\bibitem{DS}G. Denham and A. Suciu, \emph{Moment-angle complexes, monomial ideals, and Massey products}, Pure Appl. Math. \textbf{3} (2007), 25-60.
\bibitem{E}M. Eisermann, \emph{Quandle coverings and their Galois correspondence}, Fund. Math. \textbf{225} (2014), no. 1, 103-168.
\bibitem{F}E.D. Farjoun, \emph{Cellular spaces, null spaces and homotopy localization}, Lecture Notes in Mathematics {\bf 1622}, Springer-Verlag, Berlin, 1996.
\bibitem{IK1}K. Iriye and D. Kishimoto \emph{Decompositions of suspensions of spaces involving polyhedral products}, Algebr. Geom. Topol. \textbf{16} (2016), 825-841.
\bibitem{IK2}K. Iriye and D. Kishimoto, \emph{Fat wedge filtrations and decomposition of polyhedral products}, to appear in Kyoto J. Math.
\bibitem{N}T. Nosaka, \emph{Central extensions of groups and adjoint groups of quandles}, \url{arXiv:1505.03077}.
%\bibitem{P}L. Paris, \emph{Lectures on Artin groups and the $K(\pi,1)$ conjecture}, Groups of exceptional type, Coxeter groups and related geometries, Springer Proc. Math. Stat., vol. 82, Springer, New Delhi, 2014, pp. 239-257.
%\bibitem{S}M. Stafa, \emph{On the fundamental group of certain polyhedral products}, J. Pure Appl. Alg. \textbf{219} (2015), 2279-2299.
\bibitem{W}J.H.C. Whitehead, \emph{On the asphericity of regions in a $3$-sphere}, Fund. Math. \textbf{32} (1939), no. 1, 149-166.
\end{thebibliography}
\end{document}